\documentclass{amsart}
\usepackage[latin1]{inputenc}
\usepackage{amssymb,amsmath,amsthm}
\usepackage{amsfonts}
\usepackage{ifthen}
\usepackage{paralist}

\newcommand{\vect}[1]{\ensuremath{\mathbf{#1}}}
\newcommand{\card}[1]{\ensuremath{\lvert{#1}\rvert}}

\DeclareMathOperator{\Pol}{Pol}

\DeclareMathOperator{\range}{Im}

\newcommand{\subf}[1][]{\ifthenelse{\equal{#1}{}}{\ensuremath{\leq}}{\ensuremath{\leq_{#1}}}}
\newcommand{\subc}[1][]{\ifthenelse{\equal{#1}{}}{\ensuremath{\preccurlyeq}}{\ensuremath{\preccurlyeq_{#1}}}}
\newcommand{\fequiv}[1][]{\ifthenelse{\equal{#1}{}}{\ensuremath{\equiv}}{\ensuremath{\equiv_{#1}}}}

\newcommand{\eqclass}[2][]{\ifthenelse{\equal{#1}{}}{\ensuremath{[{#2}]}}{\ensuremath{[{#2}]_{#1}}}}

\theoremstyle{plain}
\newtheorem{theorem}{Theorem}[section]

\newtheorem{lemma}[theorem]{Lemma}
\newtheorem{corollary}[theorem]{Corollary}

\theoremstyle{definition}
\newtheorem*{question}{Question}

\newcommand{\cl}[1]{\ensuremath{\mathcal{#1}}}

\newcommand{\FF}{\ensuremath{\mathfrak{F}}} 
\newcommand{\id}{\ensuremath{\mathrm{id}}}
\newcommand{\nset}[1]{\ensuremath{\mathbf{#1}}}

\begin{document}
\title[Submaximal clones on $\{0,1,2\}$ with finitely many relative $\mathcal{R}$-classes]{The submaximal clones on the three-element set with finitely many relative $\mathcal{R}$-classes}
\thanks{This material is based upon work supported by the Hungarian National Foundation for Scientific Research (OTKA) grants no.\ T~048809 and K60148.
}
\author{Erkko Lehtonen}
\address[Erkko Lehtonen]
{University of Luxembourg \\
162a, avenue de la Faïencerie \\
L--1511 Luxembourg \\
Luxembourg}
\email{erkko.lehtonen@uni.lu}
\author{\'Agnes Szendrei}
\address[\'Agnes Szendrei]
{Department of Mathematics \\
University of Colorado at Boulder \\
Campus Box 395 \\
Boulder, CO 80309-0395 \\
USA}
\address
{Bolyai Institute \\
Aradi v\'ertan\'uk tere 1 \\
H--6720 Szeged \\
Hungary }
\email{szendrei@euclid.colorado.edu}
\date{\today}
\begin{abstract}
For each clone $\cl{C}$ on a set $A$ there is an associated equivalence relation analogous to Green's $\mathcal{R}$-relation, which relates two operations on $A$ if and only if each one is a substitution instance of the other using operations from $\cl{C}$. We study the maximal and submaximal clones on a three-element set and determine which of them have only finitely many relative $\mathcal{R}$-classes.
\end{abstract}

\maketitle

\section{Introduction}

This paper is a continuation to a series of studies on how functions can be classified by their substitution instances when inner functions are taken from a given set of functions. Several variants of this idea have been employed in the study of finite functions. Harrison \cite{Harrison} identified two $n$-ary Boolean functions if they are substitution instances of each other with respect to the general linear group $\mathrm{GL}(n, \mathbb{F}_2)$ or the affine general linear group $\mathrm{AGL}(n, \mathbb{F}_2)$ where $\mathbb{F}_2$ denotes the two-element field. Wang and Williams \cite{WW} defined a Boolean function $f$ to be a \emph{minor} of another Boolean function $g$ if $f$ can be obtained by substituting to each variable of $g$ a variable, a negated variable, or a constant $0$ or $1$. Classes of Boolean functions were described in terms of forbidden minors by Wang \cite{Wang}. Variants of the notion of minor were presented for Boolean functions by Feigelson and Hellerstein \cite{FH} and Zverovich \cite{Zverovich} and, in a more general setting, for operations on finite sets by Pippenger \cite{Pippenger}.

Another occurrence of the idea of classifying functions by their substitution instances can be found in semigroup theory. Green's relation $\mathcal{R}$ on a transformation semigroup $S$ relates two transformations $f, g \in S$ if and only if $f(x) = g \bigl( h_1(x) \bigr)$ and $g(x) = f \bigl( h_2(x) \bigr)$ for some $h_1, h_2 \in S \cup \{\id\}$. Henno \cite{Henno} generalized Green's relations to Menger systems (essentially, abstract clones) and described Green's relations on the clone $\cl{O}_A$ of all operations on $A$ for every set $A$. In particular, he proved that two operations on $A$ are $\mathcal{R}$-equivalent if and only if their ranges coincide.

The notions of `minor' and `$\mathcal{R}$-equivalence' for operations on a set $A$ can be defined relative to any clone $\cl{C}$ on $A$. Namely, let $\cl{C}$ be a fixed clone on $A$, and let $f$ and $g$ be operations on $A$. Then $f$ is a $\cl{C}$-minor of $g$ if $f$ can be obtained from $g$ by substituting operations from $\cl{C}$ for the variables of $g$, and $f$ and $g$ are $\cl{C}$-equivalent if each of $f$ and $g$ is a $\cl{C}$-minor of the other. Thus, Green's relation $\mathcal{R}$ described by Henno is the same notion as $\cl{O}_A$-equivalence, and each of the various notions of minor mentioned in the first paragraph corresponds to the notion of $\cl{C}$-minor for one of the smallest clones $\cl{C}$ containing only essentially at most unary operations.

This paper focuses on the following question:
\begin{question}
For which clones $\cl{C}$ are there only finitely many $\cl{C}$-equivalence classes?
\end{question}

Let us denote the set of clones on $A$ that have this property by $\FF_A$. It is easy to see that $\FF_A$ forms an order filter on the lattice of clones on $A$. Henno's result about $\cl{O}_A$-equivalence quoted above implies that $\cl{O}_A \in \FF_A$ if and only if $A$ is finite. Thus the filter $\FF_A$ is nonempty if and only if $A$ is finite. The filter is proper if $\card{A} > 1$, since the clone of projections fails to belong to $\FF_A$. In \cite{LSdiscr} we proved that every discriminator clone on $A$ belongs to $\FF_A$; furthermore, the smallest discriminator clone on $A$ is a minimal element of $\FF_A$. Moreover, for $\card{A} = 2$, the members of $\FF_A$ are precisely the discriminator clones. This is no longer true for $\card{A} > 2$, since, for example, S\l{}upecki's clone is a member of $\FF_A$ but it is not a discriminator clone.

In order to get a better understanding of the structure of the filter $\FF_A$ for finite sets $A$ of more than two elements, it is worthwhile investigating clones near the top of the lattice of clones on $A$. In \cite{LSfilter}, we decided for each clone $\cl{C}$ on a finite set $A$ that is either a maximal clone or the intersection of maximal clones whether $\cl{C} \in \FF_A$. The next natural step in this direction is taking a look at submaximal clones. The submaximal clones on the three-element set $\{0, 1, 2\}$ are well-known (see, e.g., \cite{Lau2006}), and this fact calls for a classification of these clones according to whether they are members of the filter $\FF_{\{0,1,2\}}$. That is the very goal of the current paper.

\section{Preliminaries}

Let $A$ be a nonempty set. An \emph{operation} on $A$ is a map $f \colon A^n \to A$ for some positive integer $n$, called the \emph{arity} of $f$. The set of all $n$-ary operations on $A$ is denoted by $\cl{O}_A^{(n)}$, and the set of all operations on $A$ is denoted by $\cl{O}_A$, i.e., $\cl{O}_A = \bigcup_{n \geq 1} \cl{O}_A^{(n)}$. The $n$-ary $i$-th \emph{projection} is the operation $p_i^{(n)}$ that maps every $n$-tuple $(a_1, \dotsc, a_n) \in A^n$ to its $i$-th component $a_i$. For $f \in \cl{O}_A^{(n)}$ and $g_1, \dotsc, g_n \in \cl{O}_A^{(m)}$, the \emph{composition} of $f$ with $g_1, \dotsc, g_n$ is the $m$-ary operation $f(g_1, \dotsc, g_n)$ defined by
\[
f(g_1, \dotsc, g_n)(\vect{a}) = f \bigl( g_1(\vect{a}), \dotsc, g_n(\vect{a}) \bigr) \qquad \text{for all $\vect{a} \in A^m$.}
\]
Every function $h \colon A^n \to A^m$ is uniquely determined by the $m$-tuple of operations $\vect{h} = (h_1, \dotsc, h_m)$ where $h_i = p_i^{(m)} \circ h \colon A^n \to A$ ($1 \leq i \leq m$). From now on, we will identify each function $h \colon A^n \to A^m$ with the corresponding $m$-tuple $\vect{h} = (h_1, \dotsc, h_m) \in (\cl{O}_A^{(n)})^m$ of $n$-ary operations.

A \emph{clone} on $A$ is a subset $\cl{C} \subseteq \cl{O}_A$ that contains all projections and is closed under composition. The clones on $A$ form a complete lattice under inclusion. Therefore, for each set $F \subseteq \cl{O}_A$ of operations there exists a smallest clone that contains $F$, which will be denoted by $\langle F \rangle$ and will be referred to as the \emph{clone generated by $F$.} The \emph{$n$-ary part} of a clone $\cl{C}$ is the set $\cl{C}^{(n)} = \cl{C} \cap \cl{O}_A^{(n)}$.

Let $\rho \subseteq A^r$ be a relation. The $n$-th \emph{direct power} of $\rho$ is the $r$-ary relation on $A^n$ defined by
\[
\bigl( (a_{11}, a_{12}, \dotsc, a_{1n}), (a_{21}, a_{22}, \dotsc, a_{2n}), \dotsc, (a_{r1}, a_{r2}, \dotsc, a_{rn}) \bigr) \in \rho^n
\]
if and only if $(a_{1i}, a_{2i}, \dotsc, a_{ri}) \in \rho$ for all $i \in \{1, \dotsc, n\}$. If $(\vect{a}_1, \vect{a}_2, \dotsc, \vect{a}_r) \in \rho^n$, we also say that the $n$-tuples $\vect{a}_1, \vect{a}_2, \dotsc, \vect{a}_r$ are \emph{coordinatewise $\rho$-related.}

We say that an operation $f \in \cl{O}_A^{(n)}$ \emph{preserves} an $r$-ary relation $\rho$ on $A$ (or $\rho$ is an \emph{invariant} of $f$, or $f$ is a \emph{polymorphism} of $\rho$), if for all $(a_{1i}, a_{2i}, \dotsc, a_{ri}) \in \rho$, $i = 1, \dotsc, n$, it holds that
\[
\bigl( f(a_{11}, a_{12}, \dotsc, a_{1n}), f(a_{21}, a_{22}, \dotsc, a_{2n}), \dotsc, f(a_{r1}, a_{r2}, \dotsc, a_{rn}) \bigr) \in \rho,
\]
in other words, $\bigl( f(\vect{a}_1), f(\vect{a}_2), \dotsc, f(\vect{a}_r) \bigr) \in \rho$ whenever the $n$-tuples $\vect{a}_1, \vect{a}_2, \dotsc, \vect{a}_r$ are coordinatewise $\rho$-related. We will say that $\vect{f} = (f_1, \dotsc, f_m) \in (\cl{O}_A^{(n)})^m$ \emph{preserves} an $r$-ary relation $\rho$ on $A$ if each $f_i$ ($1 \leq i \leq m$) does; that is
\[
(\vect{a}_1, \dotsc, \vect{a}_r) \in \rho^n
\quad \Rightarrow \quad
\bigl( f(\vect{a}_1), \dotsc, f(\vect{a}_r) \bigr) \in \rho^m \quad \text{for all $\vect{a}_1, \dotsc, \vect{a}_r \in A^n$.}
\]
The set of all operations on $A$ preserving a relation $\rho$ is denoted by $\Pol \rho$. For a family $R$ of relations on $A$, we denote $\Pol R = \bigcap_{\rho \in R} \Pol \rho$. For any family $R$ of relations on $A$, $\Pol R$ is a clone on $A$, and it is a well-known fact that if $A$ is finite, then every clone on $A$ is of the form $\Pol R$ for some family $R$ of relations on $A$. For general background on clones, see \cite{Lau2006,PK,Szendrei}.

Let $\cl{C}$ be a fixed clone on $A$. For arbitrary operations $f \in \cl{O}_A^{(n)}$ and $g \in \cl{O}_A^{(m)}$ we say that 
\begin{itemize}
\item
$f$ is a \emph{$\cl{C}$-minor} of $g$, in symbols $f \subf[\cl{C}] g$, if $f = g \circ \vect{h}$ for some $\vect{h} \in (\cl C^{(n)})^m$; 
\item
$f$ and $g$ are \emph{$\cl{C}$-equivalent,} in symbols $f \fequiv[\cl{C}] g$, if $f \subf[\cl{C}] g$ and $g \subf[\cl{C}] f$.
\end{itemize}
The relation $\subf[\cl{C}]$ is a quasiorder on $\cl{O}_A$, $\fequiv[\cl{C}]$ is an equivalence relation on $\cl{O}_A$, ${\subf[\cl{C}]} \subseteq {\subf[\cl{C}']}$ if and only if $\cl{C} \subseteq \cl{C}'$, and ${\fequiv[\cl{C}]} \subseteq {\fequiv[\cl{C}']}$ whenever $\cl{C} \subseteq \cl{C}'$.

Denote by $\FF_A$ the set of clones $\cl{C}$ on $A$ that have the property that there are only a finite number of $\fequiv[\cl{C}]$-classes. As discussed in the Introduction, the set $\FF_A$ forms an order filter in the lattice of clones on $A$.

Throughout this paper, we will denote the three-element set $\{0, 1, 2\}$ by $\nset{3}$.  In the following sections, we will classify the maximal and submaximal clones on $\nset{3}$ according to whether they are members of the filter $\FF_{\nset{3}}$.

\section{Maximal clones on $\nset{3}$ and their intersections}

In this section we will present a classification of the maximal clones on $\nset{3}$ according to whether they are members of $\FF_\nset{3}$. This is a special case of the general classification of maximal clones on finite sets that we obtained in \cite{LSfilter}. For the sake of easy reference, and without proof, we will also collect here some of our earlier results from \cite{ULM,LSdiscr,LSfilter}, which will be useful in the following section where we classify the submaximal clones on $\nset{3}$ accordingly.

Rosenberg completely described the maximal clones on finite sets as follows.

\begin{theorem}[Rosenberg \cite{Rosenberg}]
\label{thm:Rosenberg}
Let $A$ be a finite set with $\card{A} \geq 2$. A clone on $A$ is maximal if and only if it is of the form $\Pol \rho$, where $\rho$ is a relation on $A$ of one of the following six types:
\begin{enumerate}[\rm (1)]
\item bounded partial order,
\item prime permutation,
\item nontrivial equivalence relation,
\item prime affine relation,
\item central relation,
\item $h$-regular relation.
\end{enumerate}
\end{theorem}

Here a partial order is called \emph{bounded} if it has both a least and a greatest element. A \emph{prime permutation} is (the graph of) a fixed point free permutation on $A$ in which all cycles are of the same prime length. A \emph{prime affine relation} on $A$ is the graph of the ternary operation $x - y + z$ for some elementary abelian $p$-group $(A; +, -, 0)$ on $A$ ($p$ prime). An equivalence relation on $A$ is \emph{nontrivial} if it is neither the equality relation on $A$ nor the full relation on $A$.

An $r$-ary relation $\rho$ on $A$ is called \emph{totally reflexive} if $\rho$ contains all $r$-tuples from $A^r$ whose coordinates are not pairwise distinct, and it is called \emph{totally symmetric} if $\rho$ is invariant under any permutation of its coordinates. We say that $\rho$ is a \emph{central relation} on $A$ if $\emptyset \neq \rho \neq A^r$, $\rho$ is totally reflexive and totally symmetric and there exists an element $c \in A$ such that $\{c\} \times A^{r-1} \subseteq \rho$. The elements $c$ with this property are called the \emph{central elements} of $\rho$. Note that the arity $r$ of a central relation on $A$ satisfies $1 \leq r \leq \card{A} - 1$, and the unary central relations are just the nonempty proper subsets of $A$.

For an integer $h \geq 3$, a family $T = \{\theta_1, \dotsc, \theta_r\}$ ($r \geq 1$) of equivalence relations on $A$ is called \emph{$h$-regular} if each $\theta_i$ ($1 \leq i \leq r$) has exactly $h$ blocks, and for arbitrary blocks $B_i$ of $\theta_i$ ($1 \leq i \leq r$) the intersection $\bigcap_{1 \leq i \leq r} B_i$ is nonempty. To each $h$-regular family $T = \{\theta_1, \dotsc, \theta_r\}$ of equivalence relations on $A$ we associate an $h$-ary relation $\lambda_T$ on $A$ as follows:
\begin{multline*}
\lambda_T =
\{(a_1, \dotsc, a_h) \in A^h : \text{for each $i$, $a_1, \dotsc, a_h$ is not a transversal} \\
\text{for the blocks of $\theta_i$}\}.
\end{multline*}
Relations of the form $\lambda_T$ are called \emph{$h$-regular} (or \emph{$h$-regularly generated}) \emph{relations.} It is clear from the definition that $h$-regular relations are totally reflexive and totally symmetric.

The fact that there are exactly 18 maximal clones on $\nset{3}$ was first proved by Yablonsky \cite{Jablonskii}---this is a special case of Rosenberg's Theorem \ref{thm:Rosenberg}. The maximal clones on $\nset{3}$ are enumerated in Table \ref{table-maxcl3}, where $n_i(\cl{C})$ denotes the number of clones presented in line $i$. We also indicate for each clone whether it is a member of $\FF_\nset{3}$ (see Corollary \ref{cor:max3}). We will use the following notation. Let $\{a, b, c\} = \nset{3}$.
\begin{compactitem}
\item $\pi_3^{abc}$ denotes the 3-cycle $(abc)$, $\pi_3^{ab}$ denotes the transposition $(ab)$ on $\nset{3}$, $\pi_2^{ab}$ denotes the transposition $(ab)$ on the $2$-element set $\{a, b\}$.
\item $\epsilon_3^{ab|c}$ denotes the equivalence relation on $\nset{3}$ with $2$-element block $\{a, b\}$ and $1$-element block $\{c\}$.
\item $\leq_3^{abc}$ denotes the total order $a \leq b \leq c$ on $\nset{3}$; $\leq_2^{ab}$ denotes the total order $a \leq b$ on the $2$-element set $\{a, b\}$.
\item $\gamma_3^a$ denotes the unique central relation on $\nset{3}$ with central element $a$.
\item $\lambda_3$ denotes the unique affine relation on $\nset{3}$, $\lambda_2^{ab}$ denotes the unique affine relation on the $2$-element set $\{a, b\}$.
\item $\iota_3^3$ denotes the unique $3$-regular relation on $\nset{3}$.
\end{compactitem}

\begin{table}
\begin{tabular}{|c|l|c|c|}
\hline
$i$ & \multicolumn{1}{c|}{$\cl{C}$} & $n_i(\cl{C})$ & $\cl{C} \stackrel{?}{\in} \FF_\nset{3}$ \\
\hline
$1$ & $\Pol \{a\}$ & $3$ & yes \\
$2$ & $\Pol \{a, b\}$ & $3$ & yes \\
$3$ & $\Pol \pi_3^{012}$ & $1$ & yes \\
$4$ & $\Pol \epsilon_3^{ab|c}$ & $3$ & yes \\
$5$ & $\Pol \leq_3^{abc}$ & $3$ & no \\
$6$ & $\Pol \gamma_3^a$ & $3$ & yes \\
$7$ & $\Pol \lambda_3$ & $1$ & no \\
$8$ & $\Pol \iota_3^3$ & $1$ & yes \\
\hline
\end{tabular}

\medskip
\caption{The 18 maximal clones on the three-element set $\nset{3}$ and their membership in $\FF_\nset{3}$.}
\label{table-maxcl3}
\end{table}

\begin{theorem}[from \cite{ULM}]
\label{thm:linmon}
Let $A$ be a finite set with $\card{A} \geq 2$. If $\rho$ is a bounded partial order or a prime affine relation on $A$, then $\Pol \rho \notin \FF_A$.
\end{theorem}

The \emph{discriminator function} on $A$ is the ternary operation $t_A$ defined as follows:
\[
t_A(x, y, z) =
\begin{cases}
z, & \text{if $x = y$,} \\
x, & \text{otherwise.}
\end{cases}
\]
If a clone $\cl{C}$ on $A$ contains the discriminator function $t_A$, then $\cl{C}$ is called a \emph{discriminator clone.}
\begin{theorem}[from \cite{LSdiscr}]
\label{thm:discr}
If a clone $\cl{C}$ on a finite set $A$ contains the discriminator function $t_A$, then $\cl{C} \in \FF_A$. Moreover, the smallest clone on $A$ containing the discriminator function is a minimal member of $\FF_A$. Furthermore, if $\card{A} = 2$, then the members of $\FF_A$ are precisely the discriminator clones.
\end{theorem}

\begin{theorem}[from \cite{LSfilter}]
\label{thm:eps}
Let $A$ be a finite set, and let $E$ be a set of equivalence relations on $A$, $\Gamma$ a set of permutations on $A$, and $\Sigma$ a set of nonempty subsets of $A$. The clone $\Pol (E, \Gamma, \Sigma)$ is a member of $\FF_A$ if and only if
\begin{compactenum}[\rm \quad (a)]
\item $E$ is a chain (i.e., any two members of $E$ are comparable), and
\item $\Gamma \subseteq \Pol E$.
\end{compactenum}
\end{theorem}

\begin{theorem}[from \cite{LSfilter}]
If $\rho$ is an $r$-ary central relation on a $k$-element set $A$ such that $2 \leq r \leq k-2$ ($k \geq 4$), then $\Pol \rho \notin \FF_A$.
\end{theorem}

\begin{theorem}[from \cite{LSfilter}]
\label{thm:centralk-1}\label{thm:centralk-1subset}\label{thm:centralk-1eqrel}
Let $A$ be a finite set with $k$ elements. Let $\rho$ be a $(k-1)$-ary central relation on $A$, and let $c$ be the unique central element of $\rho$.
\begin{compactenum}[\rm (i)]
\item $\Pol(\rho, \{c\}) \in \FF_A$.
\item If $S$ is a nonempty proper subset of $A$ such that $S \neq \{c\}$, then $\Pol(\rho, S) \notin \FF_A$.
\item If $E$ is a nontrivial equivalence relation on $A$, then $\Pol(\rho, E) \notin \FF_A$.
\end{compactenum}
\end{theorem}

\begin{theorem}[from \cite{LSfilter}]
Let $A$ be a finite set with $k$ elements. If $\rho$ is an $h$-regular relation on $A$ with $h < k$, then $\Pol \rho \notin \FF_A$.
\end{theorem}

Denote by $\mathcal{T}_A$ the full transformation monoid on $A$, and denote by $\mathcal{T}_A^{-}$ the submonoid of $\mathcal{T}_A$ consisting of $\id_A$ and all non-permutations. It is well-known (see \cite{Burle} and \cite{Slupecki}) that for a finite base set $A$ with $k \geq 2$ elements, there are exactly $k + 1$ clones $\cl{C}$ such that $\cl{C}^{(1)} = \mathcal{T}_A$ and they form a chain
\[
\langle \cl{O}_A^{(1)} \rangle = \cl{B}_0 \subset \cl{B}_1 \subset \cl{B}_2 \subset \dots \subset \cl{B}_{k-1} \subset \cl{B}_k = \cl{O}_A.
\]
The clones $\cl{B}_i$ are defined as follows. For $2 \leq i \leq k$, $\cl{B}_i$ consists of all essentially at most unary functions and all functions whose range contains at most $i$ elements. $\cl{B}_1$ consists of all essentially at most unary functions and all \emph{quasilinear} functions, i.e., functions having the form $g \bigl(h_1(x_1) \oplus \dotsb \oplus h_n(x_n)\bigr)$ where $h_1, \dotsc, h_n \colon A \to \{0,1\}$, $g \colon \{0,1\} \to A$ are arbitrary mappings and $\oplus$ denotes addition modulo $2$. $\cl{B}_{k-1}$ is referred to as \emph{S\l{}upecki's clone,} and it is equal to $\Pol \rho$ for the unique $k$-regular relation $\rho$ on $A$.

Szabó extended these results and showed that if $M$ is a transformation monoid on $A$ that contains $\mathcal{T}_A^{-}$, then there are exactly $k$ clones $\cl{C}$ on $A$ such that $\cl{C}^{(1)} = M$, and they form a chain
\[
\langle M \rangle \subset \mathcal{B}_1(M) \subset \mathcal{B}_2(M) \subset \dots \subset \mathcal{B}_{k-1}(M),
\]
where each $\cl{B}_i(M)$, $1 \leq i \leq k-1$ arises from $\cl{B}_i$ by omitting all operations depending on at most one variable which are outside of $\langle M \rangle$ (see \cite{Szendrei}).

\begin{theorem}[from \cite{LSfilter}]
\label{thm:Bk-1}
\label{thm:Bk-2}
If $\cl{C}$ is a clone on a $k$-elements set $A$ ($k \geq 3$) such that $\mathcal{T}_A^{-} \subseteq \cl{C}$, then $\cl{C} \in \FF_A$ if and only if $\cl{B}_{k-1}(\mathcal{T}_A^{-}) \subseteq \cl{C}$.
\end{theorem}

\begin{theorem}[from \cite{LSfilter}]
\label{basic_props4}
Let $\rho$ be a relation on a finite set $A$, let $B$ be a nonempty subset of $A$, and let $\rho_B$ be the restriction of $\rho$ to $B$. If $\Pol \rho \in \FF_A$, then $\Pol \rho_B \in \FF_B$.
\end{theorem}

These results can be summarized in the following two theorems about maximal clones on $A$ and their intersections.

\begin{theorem}[from \cite{LSfilter}]
A maximal clone $\cl{M}$ on a $k$-element set $A$ is in $\FF_A$ if and only if $\cl{M} = \Pol \rho$ where $\rho$ is either
a prime permutation,
a nontrivial equivalence relation,
a nonempty proper subset,
a $(k-1)$-ary central relation, or
a $k$-regular relation on $A$.
\end{theorem}

\begin{theorem}[from \cite{LSfilter}]
Let $\cl{M}$, $\cl{N}$ be distinct maximal clones on a finite set $A$ of $k$ elements ($k \geq 3$).
\begin{compactenum}[\rm (1)]
\item If $\cl{N} = \cl{B}_{k-1}$ is S\l{}upecki's clone, then $\cl{M} \cap \cl{N} \notin \FF_A$.
\item If $\cl{N} = \Pol \gamma_c$ where $\gamma_c$ is the $(k-1)$-ary central relation with central element $c$, then $\cl{M} \cap \cl{N} \in \FF_A$ if and only if $\cl{M} = \Pol \{c\}$.
\item If $\cl{N} = \Pol \epsilon$ for a nontrivial equivalence relation $\epsilon$ on $A$ and $\cl{M} = \Pol \rho$ where $\rho$ is a prime permutation, a nonempty proper subset, or a nontrivial equivalence relation on $A$, then $\cl{M} \cap \cl{N} \in \FF_A$ unless
\begin{itemize}
\item $\rho$ is a prime permutation such that $\rho \notin \cl{N}$, or
\item $\rho$ is an equivalence relation incomparable to $\epsilon$.
\end{itemize}
\item If $\cl{M} = \Pol \rho$ and $\cl{N} = \Pol \tau$ where $\rho$, $\tau$ are prime permutations or nonempty proper subsets of $A$, then $\cl{M} \cap \cl{N} \in \FF_A$.
\end{compactenum} 
\end{theorem}

In the particular case when $A = \nset{3}$ we obtain the following two corollaries, the first of which justifies the statements in Table \ref{table-maxcl3} about the membership of the maximal clones on $\nset{3}$ in $\FF_\nset{3}$.
 
\begin{corollary}
\label{cor:max3}
A maximal clone $\cl{M}$ on $\nset{3}$ is in $\FF_\nset{3}$ if and only if $\cl{M} = \Pol \rho$ where $\rho$ is one of $\pi_3^{abc}$, $\epsilon_3^{ab|c}$, $\{a\}$, $\{a, b\}$, $\gamma_3^a$, $\iota_3^3$ for $\{a, b, c\} = \nset{3}$.
\end{corollary}

\begin{corollary}
Let $\cl{M}$, $\cl{N}$ be two distinct maximal clones on $\nset{3} = \{a, b, c\}$ such that $\cl{M}, \cl{N} \in \FF_A$.
\begin{compactenum}[\rm (1)]
\item If $\cl{N} = \cl{B}_2 = \Pol \iota_3^3$ is S\l{}upecki's clone, then $\cl{M} \cap \cl{N} \notin \FF_A$.
\item If $\cl{N} = \Pol \gamma_3^c$, then $\cl{M} \cap \cl{N} \in \FF_A$ if and only if $\cl{M} = \Pol \{c\}$.
\item If $\cl{N} = \Pol \epsilon_3^{ab|c}$ then $\cl{M} \cap \cl{N} \in \FF_A$ if and only if $\cl{M} = \Pol S$ for a nonempty proper subset $S$ of $\nset{3}$.
\item If $\cl{M} = \Pol \rho$ and $\cl{N} = \Pol \tau$ where each of $\rho$ and $\tau$ is $\pi_3^{abc}$ or a nonempty proper subset of $A$, then $\cl{M} \cap \cl{N} \in \FF_A$.
\end{compactenum} 
\end{corollary}

\section{Submaximal clones on $\nset{3}$}

Our aim in this section is to classify the submaximal clones on the three-element set $\nset{3}$ according to whether they are members of the filter $\FF_\nset{3}$. The submaximal clones on $\nset{3}$ were determined in the papers by Machida \cite{Machida}; Marchenkov, Demetrovics, Hannák \cite{MDH}; Demetrovics, Bagyinszki \cite{DB}; and Lau \cite{Lau1982}. We enumerate these clones in Table \ref{table-submaxcl3}, where we follow the numbering used by Lau \cite[Table 14.1]{Lau2006}.\footnote{There seems to be some confusion about the number of submaximal clones on $\nset{3}$. Lau mentions in Theorem 14.1.10 of \cite{Lau2006} that this number is 158. However, only 155 clones are listed in Table 14.1 of \cite{Lau2006}. Even more confusingly, the 1982 paper by Lau \cite{Lau1982}, on which Chapter 14 of the monograph \cite{Lau2006} is based, claims that the number is 161. Which number, if any, is correct: 155, 158 or 161? The descriptions of the submaximal clones on $\nset{3}$ in \cite{Lau1982} and in \cite{Lau2006} are identical. The note that immediately precedes Theorem 14.1.10 of \cite{Lau2006} asserts that some of the submaximal clones described in the preceding theorems are in fact the same in spite of different representations. We believe that the author was unaware of this fact at the time of writing \cite{Lau1982} and counted some clones twice. The number of such clones with double representations is 6, which is exactly the difference between 161 and 155. It seems that 155 is correct, and the number 158 is an unfortunate misprint.} Each line $i$ of Table \ref{table-submaxcl3} represents $n_i(\cl{C})$ clones, corresponding to all possible choices of $a, b, c, \alpha, \beta, \gamma$ such that $\{a, b, c\} = \{\alpha, \beta, \gamma\} = \nset{3}$. The functions $\max$, $\min$ occurring in lines 28, 29 of Table \ref{table-submaxcl3} refer to the binary maximum and minimum operations with respect to the total order $\leq_3^{abc}$. We denote $\phi_3^{ab|c} \colon \nset{3} = \{a, b, c\} \to \{0, 1\}$, $a \mapsto 0$, $b \mapsto 0$, $c \mapsto 1$. The $n$-tuple $(a, a, \dotsc, a)$ ($a \in A$) will be denoted by $\bar{a}$ and its arity $n$ is understood from the context.

\setcounter{MaxMatrixCols}{25}
\begin{table}
\begin{tabular}{|c|p{72mm}|c|c|c|}
\hline
$i$ & \multicolumn{1}{c|}{$\cl{C}$} & $n_i(\cl{C})$ & $\cl{C} \stackrel{?}{\in} \FF_\nset{3}$ & Proof \\
\hline
\phantom{$0$}$1$ & $\Pol \{a\} \cap \Pol \{b\}$ & $3$ & yes & Thm \ref{thm:discr} \\
\phantom{$0$}$2$ & $\Pol \{a\} \cap \Pol \{a, b\}$ & $6$ & yes & Thm \ref{thm:discr} \\
\phantom{$0$}$3$ & $\Pol \{a\} \cap \Pol \{b, c\}$ & $3$ & yes & Thm \ref{thm:discr} \\
\phantom{$0$}$4$ & $\Pol \{a\} \cap \Pol \epsilon_3^{bc|a}$ & $3$ & yes & Thm \ref{thm:eps} \\
\phantom{$0$}$5$ & $\Pol \{a\} \cap \Pol \gamma_3^a$ & $3$ & yes & Thm \ref{thm:centralk-1} \\
\phantom{$0$}$6$ & $\Pol \{a\} \cap \Pol \leq_3^{abc}$ & $6$ & no & Thm \ref{thm:linmon} \\
\phantom{$0$}$7$ & $\Pol \{a\} \cap \Pol \lambda_3$ & $3$ & no & Thm \ref{thm:linmon} \\
\phantom{$0$}$8$ & $\Pol \{a\} \cap \Pol \pi_3^{012}$ & $1$ & yes & Thm \ref{thm:discr} \\
\phantom{$0$}$9$ & $\Pol \{a, b\} \cap \Pol \epsilon_3^{ab|c}$ & $3$ & yes & Thm \ref{thm:eps} \\
$10$ & $\Pol \{a, b\} \cap \Pol \epsilon_3^{ac|b}$ & $6$ & yes & Thm \ref{thm:eps} \\
$11$ & $\Pol \{a, b\} \cap \Pol \leq_3^{\alpha \beta \gamma}$ & $9$ & no & Thm \ref{thm:linmon} \\
$12$ & $\Pol \{a, b\} \cap \Pol \gamma_3^\alpha$ & $9$ & no & Thm \ref{thm:centralk-1subset} \\
$13$ & $\Pol \epsilon_3^{ab|c} \cap \Pol \leq_3^{abc}$ & $6$ & no & Thm \ref{thm:linmon} \\
$14$ & $\Pol \epsilon_3^{ab|c} \cap \Pol \gamma_3^a$ & $6$ & no & Thm \ref{thm:centralk-1eqrel} \\
$15$ & $\Pol \leq_3^{abc} \cap \Pol \gamma_3^\alpha$ & $9$ & no & Thm \ref{thm:linmon} \\
$16$ & $\Pol \leq_3^{abc} \cap \Pol \iota_3^3$ & $3$ & no & Thm \ref{thm:linmon} \\
$17$ & $\Pol \pi_3^{012} \cap \Pol \lambda_3$ & $1$ & no & Thm \ref{thm:linmon} \\
$18$ & $\Pol \pi_3^{ab}$ & $3$ & yes & Thm \ref{thm:discr} \\
$19$ & $\Pol \begin{pmatrix} a & a & b & a & c \\ a & b & a & c & a \end{pmatrix}$ & $3$ & no & Thm \ref{basic_props4} \\
$20$ & $\Pol \begin{pmatrix} a & a & b & a & c & b & c \\ a & b & a & c & a & c & b \end{pmatrix}$ & $3$ & no & Thm \ref{basic_props4} \\
$21$ & $\Pol \leq_2^{ab}$ & $3$ & no & Thm \ref{basic_props4} \\
$22$ & $\Pol \pi_2^{ab}$ & $3$ & yes & Thm \ref{thm:discr} \\
$23$ & $\Pol \lambda_2^{ab}$ & $3$ & no & Thm \ref{basic_props4} \\
$24$ & $\Pol \begin{pmatrix} a & a & b & b & a \\ a & b & a & b & c \end{pmatrix}$ & $6$ & yes & Lem \ref{ClonesCase24} \\
$25$ & $\Pol \begin{pmatrix} a & a & b & b & a & c & b & c \\ a & b & a & b & c & a & c & b \end{pmatrix}$ & $3$ & no & Thm \ref{basic_props4} \\
$26$ & $\Pol \begin{pmatrix} a & b & a & b & a & b & a & b \\ a & b & a & b & b & a & a & b \\ a & b & b & a & c & c & c & c \end{pmatrix}$ & $3$ & no & Lem \ref{ClonesCase26} \\
$27$ & $\Pol \begin{pmatrix} a & b & b & a & a & b & b & a & a & b \\ a & b & a & b & a & b & a & b & a & b \\ a & b & a & a & b & a & b & b & c & c \end{pmatrix}$ & $3$ & yes & Lem \ref{ClonesCase27} \\
$28$ & $\langle \{\max\} \cup \cl{O}_\nset{3}^{(1)} \rangle \subseteq \Pol \leq_3^{abc}$ & $3$ & no & Thm \ref{thm:linmon} \\
$29$ & $\langle \{\min\} \cup \cl{O}_\nset{3}^{(1)} \rangle \subseteq \Pol \leq_3^{abc}$ & $3$ & no & Thm \ref{thm:linmon} \\
$30$ & $\langle (\Pol \lambda_3)^{(1)} \rangle \subseteq \Pol \lambda_3$ & $1$ & no & Thm \ref{thm:linmon} \\
$31$ & $\Pol \begin{pmatrix} 0 & 1 & 2 & a \\ 0 & 1 & 2 & b \end{pmatrix}$ & $3$ & no & Thm \ref{basic_props4} \\
$32$ & $\Pol (\phi^{-1} \circ \pi_2^{01} \circ \phi)$ where $\phi = \phi_3^{ab|c}$ & $3$ & yes & Lem \ref{ClonesCase32} \\
$33$ & $\Pol \begin{pmatrix} 0 & 1 & 2 & a & b & a & b \\ 0 & 1 & 2 & b & a & c & c \end{pmatrix}$ & $3$ & no & Thm \ref{basic_props4} \\
\hline
\end{tabular}
\end{table}
\begin{table}
\begin{tabular}{|c|p{72mm}|c|c|c|}
\hline
$i$ & \multicolumn{1}{c|}{$\cl{C}$} & $n_i(\cl{C})$ & $\cl{C} \stackrel{?}{\in} \FF_\nset{3}$ & Proof \\
\hline
$34$ & $\Pol \begin{pmatrix} 0 & 1 & 2 & a & a & b & b & c & c & a & b \\ 0 & 1 & 2 & a & a & b & b & c & c & b & a \\ 0 & 1 & 2 & b & c & a & c & a & b & c & c \end{pmatrix}$ & $3$ & no & Lem \ref{ClonesCase34} \\
$35$ & $\Pol \begin{pmatrix} a & a & a & a & b & b & b & b & a & b & c & c & c \\ a & a & b & b & a & a & b & b & a & b & c & c & c \\ a & b & a & b & a & b & a & b & c & c & a & b & c \end{pmatrix}$ & $3$ & no & Lem \ref{ClonesCase35} \\
$36$ & $\Pol (\lambda_2^{ab} \cup \{c\}^4)$ & $3$ & no & Thm \ref{basic_props4} \\
$37$ & $\Pol (\phi^{-1} \circ \lambda_2^{01} \circ \phi)$ where $\phi = \phi_3^{ab|c}$ & $3$ & no & Thm \ref{basic_props4} \\
$38$ & $\Pol \begin{pmatrix} 0 & 1 & 2 & a & a \\ 0 & 1 & 2 & b & c \end{pmatrix}$ & $3$ & no & Thm \ref{basic_props4} \\
$39$ & $\Pol \begin{pmatrix} 0 & 1 & 2 & a & b & a & c & b \\ 0 & 1 & 2 & b & a & c & a & c \end{pmatrix}$ & $3$ & no & Thm \ref{basic_props4} \\
$40$ & $\Pol \begin{pmatrix}
a & b & a & c & a & b & a & a & b & b & a & b & c & a & a & c & c & a & c \\
b & a & c & a & a & a & b & a & b & a & b & b & a & c & a & c & a & c & c \\
c & c & b & b & a & a & a & b & a & b & b & b & a & a & c & a & c & c & c
\end{pmatrix}$ & $3$ & no & Lem \ref{ClonesCase40} \\
$41$ & $\cl{B}_2(\mathcal{T}_3^{-} \cup \{\pi_3^{ab}\})$ & $3$ & yes & Thm \ref{thm:Bk-1} \\
$42$ & $\cl{B}_2(\mathcal{T}_3^{-} \cup \{\pi_3^{012}, \pi_3^{021}\})$ & $1$ & yes & Thm \ref{thm:Bk-1}\\
$43$ & $\cl{B}_1$ & $1$ & no & Thm \ref{thm:Bk-2} \\
\hline
\end{tabular}

\medskip
\caption{The 155 submaximal clones on the three-element set $\nset{3}$ and their membership in $\FF_\nset{3}$.}
\label{table-submaxcl3}
\end{table}

\begin{theorem}
\label{thm:main}
Let $\cl{C}$ be a submaximal clone on $\nset{3}$. Then $\cl{C} \in \FF_\nset{3}$ if and only if
\begin{compactitem}
\item $\cl{C} = \Pol \, \{a\} \cap \Pol \rho$ where $\rho$ is one of $\{b\}$, $\{\alpha, \beta\}$, $\epsilon_3^{bc|a}$, $\gamma_3^a$, $\pi_3^{012}$; or
\item $\cl{C} = \Pol \, \{a, b\} \cap \Pol \rho$ where $\rho$ is a nontrivial equivalence relation on $\nset{3}$; or
\item $\cl{C} = \Pol \rho$ where $\rho$ is one of
\[
\pi_3^{ab}, \,
\pi_2^{ab}, \,
\begin{pmatrix} a & a & b & b & a \\ a & b & a & b & c \end{pmatrix}, \,
\begin{pmatrix} a & b & b & a & a & b & b & a & a & b \\ a & b & a & b & a & b & a & b & a & b \\ a & b & a & a & b & a & b & b & c & c \end{pmatrix}, \,
\phi^{-1} \circ \pi_2^{01} \circ \phi
\]
where $\phi = \phi_3^{ab|c}$; or
\item $\cl{C} = \cl{B}_2(\mathcal{T}_3^{-} \cup \{\pi_3^{ab}\})$ or $\cl{C} = \cl{B}_2(\mathcal{T}_3^{-} \cup \{\pi_3^{012}, \pi_3^{021}\})$,
\end{compactitem}
for $\{a, b, c\} = \{\alpha, \beta, \gamma\} = \nset{3}$.
\end{theorem}

\begin{proof}
Theorem \ref{thm:main} is presented in a more explicit way in Table \ref{table-submaxcl3}, where we state for each submaximal clone $\cl{C}$ on $\nset{3}$ whether $\cl{C} \in \FF_\nset{3}$. The theorem follows from the various theorems and lemmas presented in this paper, as described in full detail below. For easy reference, we indicate in Table \ref{table-submaxcl3} for each submaximal clone $\cl{C}$ the result that proves or disproves the membership of $\cl{C}$ in $\FF_\nset{3}$.

The clones in lines 6, 7, 11, 13, 15, 16, 17, 28, 29, 30 of Table \ref{table-submaxcl3} are contained in maximal clones that are nonmembers of $\FF_\nset{3}$ by Theorem \ref{thm:linmon}, and hence they are not in $\FF_\nset{3}$.

It is easy to verify that the clones in lines 1, 2, 3, 8, 18, 22 of Table \ref{table-submaxcl3} contain the discriminator function, and hence they are members of $\FF_\nset{3}$ by Theorem \ref{thm:discr}.

It follows from Theorem \ref{thm:eps} that the clones in lines 4, 9, 10 of Table \ref{table-submaxcl3} are in $\FF_\nset{3}$.
It follows from Theorem \ref{thm:centralk-1} that the clones in line 5 of Table \ref{table-submaxcl3} are in $\FF_\nset{3}$ and the clones in lines 12, 14 of Table \ref{table-submaxcl3} are not in $\FF_\nset{3}$.
By Theorem \ref{thm:Bk-1}, the clones in lines 41, 42 of Table \ref{table-submaxcl3} are in $\FF_\nset{3}$ and the clone in line 43 of Table \ref{table-submaxcl3} is not in $\FF_\nset{3}$.

We observe that if $\rho$ is one of $\leq_2^{01}$, $\lambda_2^{01}$, $\begin{pmatrix} 0 & 0 & 1 \\ 0 & 1 & 0 \end{pmatrix}$, then the clone $\Pol \rho$ on $\{0, 1\}$ does not contain the discriminator function and hence $\Pol \rho \notin \FF_{\{0, 1\}}$ by Theorem \ref{thm:discr}. Application of Theorem \ref{basic_props4} with $B = \{a,b\}$ for the clones in lines 19, 20, 21, 23, 31, 36, 38 of Table \ref{table-submaxcl3}, with $B = \{a,c\}$ for the clones in lines 25, 33, 37, and with $B = \{b,c\}$ for the clones in line 39 shows that these clones are not in $\FF_\nset{3}$.

The membership of the remaining submaximal clones in $\FF_\nset{3}$ is proved or disproved in Lemmas \ref{ClonesCase24}--\ref{ClonesCase40} that follow.
The clones in lines 24, 27, 32 of Table \ref{table-submaxcl3} are members of $\FF_\nset{3}$ by Lemmas \ref{ClonesCase24}, \ref{ClonesCase27}, \ref{ClonesCase32}, respectively.
The clones in lines 26, 34, 35, 40 of Table \ref{table-submaxcl3} are not members of $\FF_\nset{3}$ by Lemmas \ref{ClonesCase26}, \ref{ClonesCase34}, \ref{ClonesCase35}, \ref{ClonesCase40}, respectively.
\end{proof}

\begin{lemma}
\label{ClonesCase24}
Let $A = \nset{3} = \{a, b, c\}$. For the relation
\[
\rho = \begin{pmatrix} a & a & b & b & a \\ a & b & a & b & c \end{pmatrix}
\]
in line 24 of Table \ref{table-submaxcl3}, $\Pol \rho \in \FF_\nset{3}$.
\end{lemma}
\begin{proof}
Let $\cl{C} = \Pol \rho$. Observe first that every operation in $\cl{C}$ preserves the subset $\{a, b\}$. Note also that if $\vect{a} \in A^n \setminus \{a, b\}^n$, $\vect{b} \in A^n$, then $(\vect{a}, \vect{b}) \notin \rho^n$. In the following, let $f$ and $g$ be $n$-ary and $m$-ary, respectively.

\smallskip\noindent
\textit{Claim 1.} If $\range f = \range g = \range f|_{\{a, b\}} = \range g|_{\{a, b\}}$, then $f \fequiv[\cl{C}] g$.

\smallskip\noindent
\textit{Proof of Claim 1.} Let $r = \card{\range f}$, and let $\{\vect{d}_1, \dotsc, \vect{d}_r\} \subseteq \{a, b\}^n$ be a transversal of $\ker f$. Define the mapping $\vect{h} \colon A^m \to A^n$ by the rule $\vect{h}(\vect{a}) = \vect{d}_i$ if and only if $g(\vect{a}) = f(\vect{d}_i)$. It is clear that $g = f \circ \vect{h}$. Since $\{a, b\}^2 \subseteq \rho$, we have that $(\vect{d}_i, \vect{d}_j) \in \rho^n$ for all $i, j \in \{1, \dotsc, r\}$, and hence $\vect{h} \in \cl{C}^n$. Thus, $g \subf[\cl{C}] f$. A similar argument shows that $f \subf[\cl{C}] g$.
\quad $\Diamond$

\smallskip\noindent
\textit{Claim 2.} If $\range f = \range g \neq \range f|_{\{a, b\}} = \range g|_{\{a, b\}} = \{\alpha\}$, then $f \fequiv[\cl{C}] g$.

\smallskip\noindent
\textit{Proof of Claim 2.} Let $r = \card{\range f}$, and let $\{\vect{d}_1, \vect{d}_2, \dotsc, \vect{d}_r\}$ be a transversal of $\ker f$ such that $\vect{d}_1 = \bar{a}$. Define the mapping $\vect{h} \colon A^m \to A^n$ by the rule $\vect{h}(\vect{a}) = \vect{d}_i$ if and only if $g(\vect{a}) = f(\vect{d}_i)$. It is clear that $g = f \circ \vect{h}$. Let $\vect{a}, \vect{b} \in \nset{3}^m$ and $\vect{h}(\vect{a}) = \vect{d}_i$, $\vect{h}(\vect{b}) = \vect{d}_j$. Suppose $(\vect{d}_i, \vect{d}_j) \notin \rho^n$. Since $\{a\} \times \nset{3} \subseteq \rho$, we see that $\vect{d}_i \neq \bar{a} = \vect{d}_1$. By assumption, $f|_{\{a,b\}}$ is constant $\alpha$, so $\{a, b\}^n$ is contained in a single kernel class of $f$, which by our choice is represented by $\vect{d}_1$. Therefore $\vect{d}_i \notin \{a, b\}^n$. Thus $g(\vect{a}) = f(\vect{d}_i) \neq \alpha$. Since by our assumptions $g|_{\{a, b\}}$ is constant $\alpha$, we get that $\vect{a} \notin \{a, b\}^m$. Therefore $(\vect{a}, \vect{b}) \notin \rho^m$. We conclude that $\vect{h} \in \cl{C}^n$, and hence $g \subf[\cl{C}] f$. A similar argument shows that $f \subf[\cl{C}] g$.
\quad $\Diamond$

\smallskip
We say that $f \colon A^n \to A$ \emph{has property \textnormal{(P)},} if it satisfies the following condition:
\begin{itemize}
\item[(P)] $\range f = \nset{3} = \{\alpha, \beta, \gamma\}$, $\range f|_{\{a, b\}} = \{\alpha, \beta\}$, $f(\bar{a}) = \alpha$, and there are $n$-tuples $\vect{b} \in \{a, b\}^n$, $\vect{c} \in A^n$ such that $f(\vect{b}) = \beta$, $f(\vect{c}) = \gamma$ and $(\vect{b}, \vect{c}) \in \rho^n$.
\end{itemize}

\noindent
\textit{Claim 3.} If $\range f = \range g = \nset{3} = \{\alpha, \beta, \gamma\}$, $\range f|_{\{a, b\}} = \range g|_{\{a, b\}} = \{\alpha, \beta\}$, $f(\bar{a}) = g(\bar{a})$ and both $f$ and $g$ have property (P), then $f \fequiv[\cl{C}] g$.

\smallskip\noindent
\textit{Proof of Claim 3.} Let $\vect{d}_1 = \bar{a}$, $\vect{d}_2 \in \{a, b\}^n$, $\vect{d}_3 \in A^n \setminus \{a, b\}^n$ be such that $f(\vect{d}_1) = \alpha$, $f(\vect{d}_2) = \beta$, $f(\vect{d}_3) = \gamma$ and $(\vect{d}_2, \vect{d}_3) \in \rho^n$---such $n$-tuples exist by the assumption that $f$ has property (P). Define the mapping $\vect{h} \colon A^m \to A^n$ by the rule $\vect{h}(\vect{a}) = \vect{d}_i$ if and only if $g(\vect{a}) = f(\vect{d}_i)$. It is clear that $g = f \circ \vect{h}$. Let $\vect{a}, \vect{b} \in A^m$. Suppose $\bigl( \vect{h}(\vect{a}), \vect{h}(\vect{b}) \bigr) \notin \rho^n$. Since $\{a, b\}^2 \subseteq \rho$, $\{a\} \times \nset{3} \subseteq \rho$ and $(\vect{d}_2, \vect{d}_3) \in \rho$, we see that $\vect{h}(\vect{a}) = \vect{d}_3$. By the definition of $\vect{h}$, $g(\vect{a}) = f(\vect{d}_3) = \gamma$. Since by our assumptions $\range g|_{\{a, b\}} = \{\alpha, \beta\}$, we get that $\vect{a} \notin \{a, b\}^m$. Therefore $(\vect{a}, \vect{b}) \notin \rho^m$. We conclude that $\vect{h} \in \cl{C}^n$, and hence $g \subf[\cl{C}] f$. A similar argument shows that $f \subf[\cl{C}] g$.
\quad $\Diamond$

\smallskip\noindent
\textit{Claim 4.} If $\range f = \range g = \nset{3} = \{\alpha, \beta, \gamma\}$, $\range f|_{\{a, b\}} = \range g|_{\{a, b\}} = \{\alpha, \beta\}$, $f(\bar{a}) = g(\bar{a})$ and neither $f$ nor $g$ has property (P), then $f \fequiv[\cl{C}] g$.

\smallskip\noindent
\textit{Proof of Claim 4.} Let $\vect{d}_1 = \bar{a}$, $\vect{d}_2 \in \{a, b\}^n$, $\vect{d}_3 \in A^n \setminus \{a, b\}^n$ be such that $f(\vect{d}_1) = \alpha$, $f(\vect{d}_2) = \beta$, $f(\vect{d}_3) = \gamma$. Define the mapping $\vect{h} \colon A^m \to A^n$ by the rule $\vect{h}(\vect{a}) = \vect{d}_i$ if and only if $g(\vect{a}) = f(\vect{d}_i)$. It is clear that $g = f \circ \vect{h}$. Let $\vect{a}, \vect{b} \in A^m$. Suppose $\bigl( \vect{h}(\vect{a}), \vect{h}(\vect{b}) \bigr) \notin \rho^n$. Since $\{a, b\}^2 \subseteq \rho$ and $\{a\} \times \nset{3} \subseteq \rho$, we see that either $\vect{h}(\vect{a}) = \vect{d}_3$ or $\vect{h}(\vect{a}) = \vect{d}_2$ and $\vect{h}(\vect{b}) = \vect{d}_3$. In the former case, $g(\vect{a}) = f(\vect{d}_3) = \gamma$ by the definition of $\vect{h}$. By our assumption that $\range g|_{\{a, b\}} = \{\alpha, \beta\}$, we get that $\vect{a} \in A^m \setminus \{a, b\}^m$, and hence $(\vect{a}, \vect{b}) \notin \rho^m$. In the latter case, $g(\vect{a}) = f(\vect{d}_2) = \beta$ and $g(\vect{b}) = f(\vect{d}_3) = \gamma$ by the definition of $\vect{h}$. By our assumption that $g$ does not have property (P), we get that $(\vect{a}, \vect{b}) \notin \rho^m$. We conclude that $\vect{h} \in \cl{C}^n$, and hence $g \subf[\cl{C}] f$. A similar argument shows that $f \subf[\cl{C}] g$.
\quad $\Diamond$

\smallskip
Every operation $f$ falls into one of the types prescribed in Claims 1--4:
\begin{compactitem}
\item $\range f = \range f|_{\{a,b\}}$,
\item $\range f \neq \range f|_{\{a,b\}} = \{\alpha\}$,
\item $\range f = \nset{3}$, $\range f|_{\{a,b\}} = \{\alpha, \beta\}$ and $f$ has property (P),
\item $\range f = \nset{3}$, $\range f|_{\{a,b\}} = \{\alpha, \beta\}$ and $f$ does not have property (P),
\end{compactitem}
and there are only finitely many possibilities for $\range f$, $\range f|_{\{a,b\}}$ and $f(\bar{a})$. We conclude that there are only a finite number of $\fequiv[\cl{C}]$-classes.
\end{proof}

\begin{lemma}
\label{ClonesCase26}
Let $A = \nset{3} = \{a, b, c\}$. For the relation
\[
\rho = \begin{pmatrix} a & b & a & b & a & b & a & b \\ a & b & a & b & b & a & a & b \\ a & b & b & a & c & c & c & c \end{pmatrix}
\]
in line 26 of Table \ref{table-submaxcl3}, $\Pol \rho \notin \FF_\nset{3}$.
\end{lemma}
\begin{proof}
Let $\cl{C} = \Pol \rho$. For $1 \leq i \leq n$, denote by $\vect{e}_i^n$ the $n$-tuple whose $i$-th component is $a$ and the other components are $b$. For $1 \leq i \leq n - 1$, denote by $\vect{d}_i^n$ the $n$-tuple
\[
(b, \dotsc, b, \underset{i}{c}, \underset{i+1}{c}, b, \dotsc, b)
\]
and denote by $\vect{d}_n^n$ the $n$-tuple $(c, b, b, \dots, b, c)$.

For $n \geq 3$, define the operation $f_n \colon A^n \to A$ as follows:
\[
f_n(\vect{a}) =
\begin{cases}
2, & \text{if $\vect{a} = \vect{e}_1^n$,} \\
1, & \text{if $\vect{a} = \vect{e}_i^n$ for some $i \in \{2, \dotsc, n\}$,} \\
1, & \text{if $\vect{a} = \vect{d}_i^n$ for some $i \in \{1, \dotsc, n - 1\}$,} \\
2, & \text{if $\vect{a} = \vect{d}_n^n$,} \\
0, & \text{otherwise.}
\end{cases}
\]
We claim that $f_n \not\fequiv[\cl{C}] f_m$ whenever $n \neq m$, and hence there are infinitely many $\fequiv[\cl{C}]$-classes. For, let $n < m$, and suppose on the contrary that there exists a map $\vect{h} \in (\cl{C}^{(n)})^m$ such that $f_n = f_m \circ \vect{h}$. Since every operation in $\cl{C}$ preserves $\{a, b\}$, $\vect{h}$ maps $\{a, b\}^n$ into $\{a, b\}^m$. Thus, there is a map $\tau \colon \{1, \dotsc, n\} \to \{1, \dotsc, m\}$ such that $\tau(1) = 1$, $\tau(i) \neq 1$ for $i \neq 1$ and $\vect{h}(\vect{e}_i^n) = \vect{e}_{\tau(i)}^m$ for all $i \in \{1, \dotsc, n\}$.

We have that $\vect{h}(\vect{d}_n^n) \in \{\vect{e}_1^m, \vect{d}_m^m\}$. Suppose that $\vect{h}(\vect{d}_n^n) = \vect{e}_1^m$. Then $\bigl( \bar{b}, \vect{e}_n^n, \vect{d}_n^n \bigr) \in \rho^n$, but the $m$-tuples $\vect{h}(\bar{b})$, $\vect{h}(\vect{e}_n^n) = \vect{e}_{\tau(n)}^m$, $\vect{h}(\vect{d}_n^n) = \vect{e}_1^m$ are all in $\{a, b\}^m$ and $\vect{h}(\bar{b}) \neq \vect{e}_{\tau(n)}^m$ since $f_m \bigl( \vect{h}(\bar{b}) \bigr) = f_n(\bar{b}) = 0$, $f_m(\vect{e}_{\tau(n)}^m) = 1$. Hence $\bigl( \vect{h}(\bar{b}), \vect{h}(\vect{e}_n^n), \vect{h}(\vect{d}_n^n) \bigr) \notin \rho^m$, which contradicts the assumption that $\vect{h} \in \cl{C}^m$. Thus, $\vect{h}(\vect{d}_n^n) = \vect{d}_m^m$.

For each $i$ ($1 \leq i \leq n - 1$), we have that $\vect{h}(\vect{d}_i^n) \in \{\vect{e}_2^m, \dotsc, \vect{e}_m^m, \vect{d}_1^m, \dotsc, \vect{d}_{m-1}^m\}$. Suppose that there is an $i \in \{1, \dotsc, n-1\}$ such that $\vect{h}(\vect{d}_i^n) = \vect{e}_j^m$ for some $j \in \{2, \dotsc, m\}$. Then $\bigl( (\bar{b}), \vect{e}_i^n, \vect{d}_i^n \bigr) \in \rho^n$, but the $m$-tuples $\vect{h}(\bar{b})$, $\vect{h}(\vect{e}_i^n) = \vect{e}_{\tau(i)}^m$, $\vect{h}(\vect{d}_i^n) = \vect{e}_j^m$ are all in $\{a, b\}^m$ and $\vect{h}(\bar{b}) \neq \vect{e}_{\tau(i)}^m$ since $f_m \bigl( \vect{h}(\bar{b}) \bigr) = f_n(\bar{b}) = 0$, $f_m(\vect{e}_{\tau(i)}^m) \neq 0$. Hence $\bigl( \vect{h}(\bar{b}), \vect{h}(\vect{e}_i^n), \vect{h}(\vect{d}_i^n) \bigr) \notin \rho^m$, which contradicts the assumption that $\vect{h} \in \cl{C}^m$. We conclude that there exists a map $\nu \colon \{1, \dotsc, n\} \to \{1, \dotsc, m\}$ such that $\nu(n) = m$, $\nu(i) \neq m$ for $i \neq n$ and $\vect{h}(\vect{d}_i^n) = \vect{d}_{\nu(i)}^m$ for all $i \in \{1, \dotsc, n\}$.

It is easy to verify that for all $p \geq 3$, $(\vect{e}_i^p, \vect{e}_j^p, \vect{d}_\ell^p) \in \rho^p$ if and only if $\{i, j\} \subseteq \{\ell, \ell+1\}$ and $\ell < p$ or $\{i, j\} \subseteq \{1, p\}$ and $\ell = p$. Since $(\vect{e}_1^n, \vect{e}_1^n, \vect{d}_1^n) \in \rho^n$ and $\vect{h}(\vect{e}_1^n) = \vect{e}_1^m$, we have that $(\vect{e}_1^m, \vect{e}_1^m, \vect{d}_{\nu(1)}^m) = \bigl( \vect{h}(\vect{e}_1^n), \vect{h}(\vect{e}_1^n), \vect{h}(\vect{d}_1^n) \bigr) \in \rho^m$. By the previous observation, $\nu(1) \in \{1, m\}$, but since we have that $\nu(1) \neq m$, we conclude that $\nu(1) = 1$. Similarly, $(\vect{e}_n^n, \vect{e}_n^n, \vect{d}_n^n) \in \rho$ and $\vect{h}(\vect{d}_n^n) = \vect{d}_m^m$ imply that $(\vect{e}_{\tau(n)}^m, \vect{e}_{\tau(n)}^m, \vect{d}_m^m) \in \rho^m$. It follows from the previous observation that $\tau(n) \in \{1, m\}$, but since $\tau(n) \neq 1$, we have that $\tau(n) = m$. Similarly, for $1 \leq i \leq n-1$, $(\vect{e}_i^n, \vect{e}_{i+1}^n, \vect{d}_i^n) \in \rho^n$ implies $(\vect{e}_{\tau(i)}^m, \vect{e}_{\tau(i+1)}^m, \vect{d}_{\nu(i)}^m) \in \rho^m$, and from the previous observation and the fact that $\nu(i) \neq m$ when $i \neq n$ it follows that $\{\tau(i), \tau(i+1)\} \subseteq \{\nu(i), \nu(i) + 1\}$. Thus, $\tau(i + 1) \leq \tau(i) + 1$, and hence $\tau(i) \leq i$ for all $i \in \{1, \dotsc, n\}$. Then $\tau(n) \leq n < m = \tau(n)$, and we have reached the desired contradiction.
\end{proof}

\begin{lemma}
\label{ClonesCase27}
Let $A = \nset{3} = \{a, b, c\}$. For the relation
\[
\rho =
\begin{pmatrix}
a & b & b & a & a & b & b & a & a & b \\
a & b & a & b & a & b & a & b & a & b \\
a & b & a & a & b & a & b & b & c & c
\end{pmatrix}
\]
in line 27 of Table \ref{table-submaxcl3}, $\Pol \rho \in \FF_\nset{3}$.
\end{lemma}
\begin{proof}
Let $\cl{C} = \Pol \rho$. Observe first that every operation in $\cl{C}$ preserves the subset $\{a, b\}$. Note also that $(\vect{a}, \vect{b}, \vect{c}) \in \rho^n$ if and only if $(\vect{b}, \vect{a}, \vect{c}) \in \rho^n$. Also, if $\vect{a} \notin \{a, b\}^n$ or $\vect{b} \notin \{a, b\}^n$, then $(\vect{a}, \vect{b}, \vect{c}) \notin \rho^n$. In the following, let $f$ and $g$ be $n$-ary and $m$-ary, respectively.

\smallskip\noindent
\textit{Claim 1.} If $\range f = \range g = \range f|_{\{a, b\}} = \range g|_{\{a, b\}}$, then $f \fequiv[\cl{C}] g$.

\smallskip\noindent
\textit{Proof of Claim 1.} Let $r = \card{\range f}$, and let $\{\vect{d}_1, \dotsc, \vect{d}_r\} \subseteq \{a, b\}^n$ be a transversal of $\ker f$. Define the mapping $\vect{h} \colon A^m \to A^n$ by the rule $\vect{h}(\vect{a}) = \vect{d}_i$ if and only if $g(\vect{a}) = f(\vect{d}_i)$. It is clear that $g = f \circ \vect{h}$. Since $\{a, b\}^3 \subseteq \rho$, we have that $(\vect{d}_i, \vect{d}_j, \vect{d}_\ell) \in \rho^n$ for all $i, j, \ell \in \{1, \dotsc, r\}$ and hence $\vect{h} \in \cl{C}^n$. Thus, $g \subf[\cl{C}] f$. A similar argument shows that $f \subf[\cl{C}] g$.
\quad $\Diamond$

\smallskip\noindent
\textit{Claim 2.} If $\range f = \range g \neq \range f|_{\{a, b\}} = \range g|_{\{a, b\}} = \{\alpha\}$, then $f \fequiv[\cl{C}] g$.

\smallskip\noindent
\textit{Proof of Claim 2.} Let $r = \card{\range f}$, and let $\{\vect{d}_1, \vect{d}_2, \dotsc, \vect{d}_r\}$  be a transversal of $\ker f$ such that $\vect{d}_1 = \bar{a}$. Define the mapping $\vect{h} \colon A^m \to A^n$ by the rule $\vect{h}(\vect{a}) = \vect{d}_i$ if and only if $g(\vect{a}) = f(\vect{d}_i)$. It is clear that $g = f \circ \vect{h}$. Let $\vect{a}, \vect{b}, \vect{c} \in A^m$, and let $\vect{h}(\vect{a}) = \vect{d}_i$, $\vect{h}(\vect{b}) = \vect{d}_j$, $\vect{h}(\vect{c}) = \vect{d}_\ell$. Suppose $(\vect{d}_i, \vect{d}_j, \vect{d}_\ell) \notin \rho^n$. Since $\{a, b\}^3 \subseteq \rho$, we have that one of $\vect{d}_i$, $\vect{d}_j$, $\vect{d}_\ell$ is not in $\{a, b\}^n$. If $\vect{d}_\ell \notin \{a, b\}^n$, then $\vect{d}_i$ and $\vect{d}_j$ cannot both be equal to $\vect{d}_1 = \bar{a}$, because $(\vect{d}_1, \vect{d}_1, \vect{d}_\ell) \in \rho^n$. By assumption, $f|_{\{a, b\}}$ is constant $\alpha$, so $\{a, b\}^n$ is contained in a single kernel class of $f$, which by our choice is represented by $\vect{d}_1$. Thus, it actually holds that $\vect{d}_i \notin \{a, b\}^n$ or $\vect{d}_j \notin \{a, b\}^n$. By the definition of $\vect{h}$, we have that $g(\vect{a}) = f(\vect{d}_i) \neq \alpha$ or $g(\vect{b}) = f(\vect{d}_j) \neq \alpha$, and by our assumption that $g|_{\{a,b\}}$ is constant $\alpha$ we get that $\vect{a} \notin \{a, b\}^m$ or $\vect{b} \notin \{a, b\}^m$. Therefore $(\vect{a}, \vect{b}, \vect{c}) \notin \rho^m$, and we conclude that $\vect{h} \in \cl{C}^n$. Hence $g \subf[\cl{C}] f$. A similar argument shows that $f \subf[\cl{C}] g$.
\quad $\Diamond$

\smallskip
We say that $f \colon A^n \to A$ \emph{has property \textnormal{(Q)},} if it satisfies the following condition:
\begin{itemize}
\item[(Q)] $\range f = \nset{3} = \{\alpha, \beta, \gamma\}$, $\range f|_{\{a, b\}} = \{\alpha, \beta\}$ and there are $n$-tuples $\vect{a}, \vect{b} \in \{a, b\}^n$, $\vect{c} \in A^n$ such that $f(\vect{a}) = \alpha$, $f(\vect{b}) = \beta$, $f(\vect{c}) = \gamma$ and $(\vect{a}, \vect{b}, \vect{c}) \in \rho^n$.
\end{itemize}

\noindent
\textit{Claim 3.} If $\range f = \range g = \nset{3} = \{\alpha, \beta, \gamma\}$, $\range f|_{\{a, b\}} = \range g|_{\{a, b\}} = \{\alpha, \beta\}$ and both $f$ and $g$ have property (Q), then $f \fequiv[\cl{C}] g$.

\smallskip\noindent
\textit{Proof of Claim 3.} Let $\vect{d}_1 \in \{a, b\}^n$, $\vect{d}_2 \in \{a, b\}^n$, $\vect{d}_3 \in A^n \setminus \{a, b\}^n$ be such that $f(\vect{d}_1) = \alpha$, $f(\vect{d}_2) = \beta$, $f(\vect{d}_3) = \gamma$ and $(\vect{d}_1, \vect{d}_2, \vect{d}_3) \in \rho^n$---such $n$-tuples exist by the assumption that $f$ has property (Q). Define the mapping $\vect{h} \colon A^m \to A^n$ by the rule that $\vect{h}(\vect{a}) = \vect{d}_i$ if and only if $g(\vect{a}) = f(\vect{d}_i)$. It is clear that $g = f \circ \vect{h}$. Let $\vect{a}, \vect{b}, \vect{c} \in A^m$. Suppose $\bigl( \vect{h}(\vect{a}), \vect{h}(\vect{b}), \vect{h}(\vect{c}) \bigr) \notin \rho^n$. Since $\{a, b\}^3 \subseteq \rho$, one of $\vect{h}(\vect{a})$, $\vect{h}(\vect{b})$, $\vect{h}(\vect{c})$ equals $\vect{d}_3$. It is not possible that $\vect{h}(\vect{c}) = \vect{d}_3$ and $\{\vect{h}(\vect{a}), \vect{h}(\vect{b})\} \subseteq \{\vect{d}_1, \vect{d}_2\}$, because on one hand $(\vect{x}, \vect{x}, \vect{y}) \in \rho^n$ for all $\vect{x} \in \{a, b\}^n$, $\vect{y} \in A^n$, and on the other hand, by our choice of representatives of kernel classes, $(\vect{d}_1, \vect{d}_2, \vect{d}_3) \in \rho^n$ and hence also $(\vect{d}_2, \vect{d}_1, \vect{d}_3) \in \rho^n$. Thus we have in fact that $\vect{h}(\vect{a}) = \vect{d}_3$ or $\vect{h}(\vect{b}) = \vect{d}_3$. By the definition of $\vect{h}$, $g(\vect{a}) = f(\vect{d}_3) = \gamma$ or $g(\vect{b}) = f(\vect{d}_3) = \gamma$. Since by our assumptions $\range g|_{\{a, b\}} = \{\alpha, \beta\}$, we get that $\vect{a} \notin \{a, b\}^m$ or $\vect{b} \notin \{a, b\}^m$, and hence $(\vect{a}, \vect{b}, \vect{c}) \notin \rho^m$. Therefore $\vect{h} \in \cl{C}^n$, and we conclude that $g \subf[\cl{C}] f$. A similar argument shows that $f \subf[\cl{C}] g$.
\quad $\Diamond$

\smallskip\noindent
\textit{Claim 4.} If $\range f = \range g = \nset{3} = \{\alpha, \beta, \gamma\}$, $\range f|_{\{a, b\}} = \range g|_{\{a, b\}} = \{\alpha, \beta\}$ and neither $f$ nor $g$ has property (Q), then $f \fequiv[\cl{C}] g$.

\smallskip\noindent
\textit{Proof of Claim 4.} Let $\vect{d}_1 \in \{a, b\}^n$, $\vect{d}_2 \in \{a, b\}^n$ and $\vect{d}_3 \in A^n \setminus \{a, b\}^n$ be $n$-tuples such that $f(\vect{d}_1) = \alpha$, $f(\vect{d}_2) = \beta$, $f(\vect{d}_3) = \gamma$. Define the mapping $\vect{h} \colon A^m \to A^n$ by the rule $\vect{h}(\vect{a}) = \vect{d}_i$ if and only if $g(\vect{a}) = f(\vect{d}_i)$. It is clear that $g = f \circ \vect{h}$. Let $\vect{a}, \vect{b}, \vect{c} \in A^m$. Suppose $\bigl( \vect{h}(\vect{a}), \vect{h}(\vect{b}), \vect{h}(\vect{c}) \bigr) \notin \rho^n$. Since $\{a, b\}^3 \subseteq \rho$, one of $\vect{h}(\vect{a})$, $\vect{h}(\vect{b})$, $\vect{h}(\vect{c})$ equals $\vect{d}_3$. If $\vect{h}(\vect{a}) = \vect{d}_3$, then we get by the definition of $\vect{h}$ that $g(\vect{a}) = f(\vect{d}_3) = \gamma$, and by the assumption that $\range g|_{\{a,b\}} = \{\alpha, \beta\}$, we have that $\vect{c} \notin \{a, b\}^m$; thus $(\vect{a}, \vect{b}, \vect{c}) \notin \rho^m$. If $\vect{h}(\vect{b}) = \vect{d}_3$, then a similar argument shows that $(\vect{a}, \vect{b}, \vect{c}) \notin \rho^m$.

Assume then that none of $\vect{h}(\vect{a})$ and $\vect{h}(\vect{b})$ equals $\vect{d}_3$ but $\vect{h}(\vect{c}) = \vect{d}_3$. We must have $\vect{h}(\vect{a}) \neq \vect{h}(\vect{b})$, for otherwise $\bigl( \vect{h}(\vect{a}), \vect{h}(\vect{b}), \vect{h}(\vect{c}) \bigr) \in \rho^n$. Assume that $\vect{h}(\vect{a}) = \vect{d}_1$ and $\vect{h}(\vect{b}) = \vect{d}_2$. By the definition of $\vect{h}$ we get that $g(\vect{a}) = f(\vect{d}_1) = \alpha$, $g(\vect{b}) = f(\vect{d}_2) = \beta$, $g(\vect{c}) = f(\vect{d}_3) = \gamma$. By the assumption that $\range g|_{\{a,b\}} = \{\alpha, \beta\}$, we have that $\vect{c} \notin \{a, b\}^m$. If $\vect{a} \notin \{a, b\}^m$ or $\vect{b} \notin \{a, b\}^m$, then $(\vect{a}, \vect{b}, \vect{c}) \notin \rho^m$, so we can assume that $\vect{a}, \vect{b} \in \{a, b\}^m$. But then the assumption that $g$ does not have property (Q) implies that $(\vect{a}, \vect{b}, \vect{c}) \notin \rho^m$. In the only remaining case when $\vect{h}(\vect{a}) = \vect{d}_2$, $\vect{h}(\vect{b}) = \vect{d}_1$, $\vect{h}(\vect{c}) = \vect{d}_3$, we can deduce in a similar way that $(\vect{a}, \vect{b}, \vect{c}) \notin \rho^m$, taking into account that $(\vect{a}, \vect{b}, \vect{c}) \in \rho^m$ if and only if $(\vect{b}, \vect{a}, \vect{c}) \in \rho^m$.

We conclude that $\vect{h} \in \cl{C}^n$, and hence $g \subf[\cl{C}] f$. A similar argument shows that $f \subf[\cl{C}] g$.
\quad $\Diamond$

\smallskip
Every operation $f$ falls into one of the types prescribed in Claims 1--4:
\begin{compactitem}
\item $\range f = \range f|_{\{a,b\}}$,
\item $\range f \neq \range f|_{\{a,b\}} = \{\alpha\}$,
\item $\range f = \nset{3}$, $\range f|_{\{a,b\}} = \{\alpha, \beta\}$ and $f$ has property (Q),
\item $\range f = \nset{3}$, $\range f|_{\{a,b\}} = \{\alpha, \beta\}$ and $f$ does not have property (Q),
\end{compactitem}
and there are only finitely many possibilities for $\range f$ and $\range f|_{\{a,b\}}$. We conclude that there are only a finite number of $\fequiv[\cl{C}]$-classes.
\end{proof}

\begin{lemma}
\label{ClonesCase32}
Let $A = \nset{3} = \{a, b, c\}$, and let $\phi \colon \nset{3} \to \{0, 1\}$ be the map $a \mapsto 0$, $b \mapsto 0$, $c \mapsto 1$. For the relation $\rho = \phi^{-1} \circ \pi_2^{01} \circ \phi$ in line 32 of Table \ref{table-submaxcl3}, $\Pol \rho \in \FF_\nset{3}$.
\end{lemma}
\begin{proof}
Let $\cl{C} = \Pol \rho$. We may think of the relation $\rho$ as a transposition of the two blocks of the equivalence relation $\epsilon_3^{ab|c}$. For notational simplicity, let $\sigma = \epsilon_3^{ab|c}$. (More precisely, this relation is the full inverse image, under the natural map $A \to A / \sigma$, of the transposition of the two elements of $A / \sigma$.) If $D$ is a block of $\sigma$, let $D'$ denote its complement (i.e., its image under the transposition of the two blocks).

For each $n$, $\sigma^n$ partitions $A^n$ into blocks of the form $B = B_1 \times B_2 \times \dotsb \times B_n$ where each $B_i$ is $\{a, b\}$ or $\{c\}$. Let $B'$ denote the block $B_1' \times B_2' \times \dotsb \times B_n'$ of $\sigma^n$.

\smallskip\noindent
\textit{Claim.} If $f$, $g$ are operations on $A$, say $f$ is $m$-ary and $g$ is $n$-ary, such that for every block $B$ of $\sigma^n$ on $A^n$ there is a block $C$ of $\sigma^m$ on $A^m$ such that
\[
\range g|_{B} \subseteq \range f|_{C} \quad \text{and} \quad \range g|_{B'} \subseteq \range f|_{C'},
\]
then there exists $\vect{h} \in \cl{C}^m$ such that $g = f \circ \vect{h}$.

\smallskip\noindent
\textit{Proof of Claim.}
$A^n$ is partitioned into disjoint sets of the form $B \cup B'$ with $B$ as above. For each such set choose $C$ according to the assumption. Then there exist $h_B \colon B \to C$ and $h_B' \colon B' \to C'$ such that $f|_C \circ h_B = g|_B$ and $f|_C' \circ h_B' = g|_B'$. Let $\vect{h}$ be the union of all $h_B \cup h_B'$. It is easy to see that $\vect{h}$ preserves $\rho$ and $f \circ \vect{h} = g$.
\quad $\Diamond$

\smallskip\noindent
\textit{Corollary.} If $f$, $g$ are operations on $A$, say $f$ is $m$-ary and $g$ is $n$-ary, such that
\[
\{(\range g|_{B}, \range g|_{B'}) : \text{$B$ is a block of $\sigma^n$ on $A^n$}\}
\]
equals
\[
\{(\range f|_{C}, \range f|_{C'}) : \text{$C$ is a block of $\sigma^m$ on $A^m$}\},
\]
then $f$ and $g$ are $\cl{C}$-equivalent.

\smallskip
Since both sets above are subsets of $\mathcal{P}(A) \times \mathcal{P}(A)$, which is finite, it follows that there are only a finite number of $\fequiv[\cl{C}]$-classes.
\end{proof}

\begin{lemma}
\label{ClonesCase34}
Let $A = \nset{3} = \{a, b, c\}$. For the relation
\[
\rho =
\begin{pmatrix}
0 & 1 & 2 & a & a & b & b & c & c & a & b \\
0 & 1 & 2 & a & a & b & b & c & c & b & a \\
0 & 1 & 2 & b & c & a & c & a & b & c & c
\end{pmatrix}
\]
in line 34 of Table \ref{table-submaxcl3}, $\Pol \rho \notin \FF_\nset{3}$.
\end{lemma}
\begin{proof}
Let $\cl{C} = \Pol \rho$. For $n \geq 3$, define the operation $f_n \colon A^{n+1} \to A$ as follows:
\[
f_n(\vect{a}) =
\begin{cases}
0, & \text{if $\vect{a} \in \{a\} \times \{c\} \times \{a, b\}^{n-1}$,} \\
1, & \text{if $\vect{a} \in \{b\} \times \{c\} \times \{a, b\}^{n-1}$,} \\
1, & \text{if $\vect{a} \in \{a, b\}^i \times \{a\} \times \{c\} \times \{a, b\}^{n - i -1}$ for some $i \in \{1, \dotsc, n-2\}$,} \\
2, & \text{if $\vect{a} \in \{a, b\}^i \times \{b\} \times \{c\} \times \{a, b\}^{n - i -1}$ for some $i \in \{1, \dotsc, n-2\}$,} \\
2, & \text{if $\vect{a} \in \{a, b\}^{n - 1} \times \{a\} \times \{c\}$,} \\
0, & \text{if $\vect{a} \in \{a, b\}^{n - 1} \times \{b\} \times \{c\}$,} \\
0, & \text{otherwise.}
\end{cases}
\]
We claim that $f_n \not\fequiv[\cl{C}] f_m$ whenever $n \neq m$ and hence there are infinitely many $\fequiv[\cl{C}]$-classes. For, let $n < m$ and assume on the contrary that there exists a map $\vect{h} \in \cl{C}^m$ such that $f_n = f_m \circ \vect{h}$.

Note that every operation in $\cl{C}$ preserves the equivalence relation $\epsilon_3^{ab|c}$. For notational simplicity, let $\sigma = \epsilon_3^{ab|c}$. For each $n$, $\sigma^n$ partitions $A^n$ into blocks of the form $B_1 \times B_2 \times \dotsb \times B_n$ where each $B_i$ is either $\{a, b\}$ or $\{c\}$. Thus, $\vect{h}$ maps each $\sigma^n$-block $C$ of $A^n$ into a $\sigma^m$-block $C'$ of $A^m$. This implies that for $1 \leq i \leq n$,  $\alpha_1, \dotsc, \alpha_{n+1}, \beta \in \{a, b\}$,
\[
\vect{h}(\alpha_1, \dotsc, \alpha_{i-1}, \beta, c, \alpha_{i+2}, \dotsc, \alpha_{n+1}) \in \{a, b\}^{\tau(i) - 1} \times \{\beta\} \times \{c\} \times \{a, b\}^{n - \tau(i)}
\]
for some $\tau \colon \{1, \dotsc, n\} \to \{1, \dotsc, m\}$ such that $\tau(1) = 1$ and $\tau(n) = m$.

For $2 \leq i \leq n - 1$, the $(n+1)$-tuples
\[
(a, \dotsc, a, \underset{i}{a}, \underset{i+1}{c}, a, \dotsc, a), \quad
(a, \dotsc, a, \underset{i}{b}, \underset{i+1}{c}, a, \dotsc, a), \quad
(a, \dotsc, a, \underset{i}{c}, \underset{i+1}{a}, a, \dotsc, a),
\]
are coordinatewise $\rho$-related. Thus, their images by $\vect{h}$, namely
\begin{align*}
&(\alpha_1, \dotsc, \alpha_{\tau(i+1)-1}, a, c, \alpha_{\tau(i+1)+2}, \dotsc, \alpha_{n+1}), \\
&(\beta_1, \dotsc, \beta_{\tau(i+1)-1}, b, c, \beta_{\tau(i+1)+2}, \dotsc, \beta_{n+1}), \\
&(\gamma_1, \dotsc, \gamma_{\tau(i)-1}, a, c, \gamma_{\tau(i)+2}, \dotsc, \gamma_{n+1}),
\end{align*}
for some $\alpha_i$'s, $\beta_i$'s, $\gamma_i$'s in $\{a, b\}$, are coordinatewise $\rho$-related as well. But this is only possible if $\tau(i + 1) = \tau(i) + 1$ for all $i \in \{1, \dotsc, n - 1\}$. Since $\tau(1) = 1$, it follows that $\tau(n) = n < m = \tau(m)$, and we have reached the desired contradiction.
\end{proof}

\begin{lemma}
\label{ClonesCase35}
Let $A = \nset{3} = \{a, b, c\}$. For the relation
\[
\rho =
\begin{pmatrix}
a & a & a & a & b & b & b & b & a & b & c & c & c \\
a & a & b & b & a & a & b & b & a & b & c & c & c \\
a & b & a & b & a & b & a & b & c & c & a & b & c
\end{pmatrix}
\]
in line 35 of Table \ref{table-submaxcl3}, $\Pol \rho \notin \FF_\nset{3}$.
\end{lemma}
\begin{proof}
Let $\cl{C} = \Pol \rho$. For $n \geq 3$, $1 \leq i \leq n$, $\alpha, \beta \in \{a, b\}$ denote by $\vect{d}_{i, \alpha \beta}^n$ the $(n + 1)$-tuple
\[
(c, \dotsc, c, \underset{i}{\alpha}, \underset{i+1}{\beta}, c, \dotsc, c).
\]
For $n \geq 3$, define the operation $f_n \colon A^{n+1} \to A$ as follows:
\[
f_n(\vect{a}) =
\begin{cases}
0, & \text{if $\vect{a} = \vect{d}_{1, a \beta}^n$ for some $\beta \in \{a, b\}$,} \\
1, & \text{if $\vect{a} = \vect{d}_{1, b \beta}^n$ for some $\beta \in \{a, b\}$,} \\
1, & \text{if $\vect{a} = \vect{d}_{i, a \beta}^n$ for some $i \in \{2, \dotsc, n - 1\}$, $\beta \in \{a, b\}$,} \\
2, & \text{if $\vect{a} = \vect{d}_{i, b \beta}^n$ for some $i \in \{2, \dotsc, n - 1\}$, $\beta \in \{a, b\}$,} \\
2, & \text{if $\vect{a} = \vect{d}_{n, a \beta}^n$ for some $\beta \in \{a, b\}$,} \\
0, & \text{if $\vect{a} = \vect{d}_{n, b \beta}^n$ for some $\beta \in \{a, b\}$,} \\
0, & \text{otherwise.}
\end{cases}
\]
We claim that $f_n \not\fequiv[\cl{C}] f_m$ whenever $n \neq m$ and hence there are infinitely many $\fequiv[\cl{C}]$-classes. For, let $n < m$ and assume on the contrary that there exists a map $\vect{h} \in \cl{C}^m$ such that $f_n = f_m \circ \vect{h}$.

Note that every operation in $\cl{C}$ preserves the equivalence relation $\epsilon_3^{ab|c}$. For notational simplicity, let $\sigma = \epsilon_3^{ab|c}$. For each $n$, $\sigma^n$ partitions $A^n$ into blocks of the form $B_1 \times B_2 \times \dotsb \times B_n$ where each $B_i$ is either $\{a, b\}$ or $\{c\}$. Thus, $\vect{h}$ maps each $\sigma^n$-block $C$ of $A^n$ into some $\sigma^m$-block $C'$ of $A^m$. Observe that for every $p \geq 3$,
\begin{compactitem}
\item the only $\sigma^p$-block $C$ of $A^p$ such that $\range f_p|_C = \{0, 1\}$ is the block of $\vect{d}_{1,aa}^p$,
\item the only $\sigma^p$-block $C$ of $A^p$ such that $\range f_p|_C = \{0, 2\}$ is the block of $\vect{d}_{p,aa}^p$,
\item the only $\sigma^p$-blocks $C$ of $A^p$ such that $\range f_p|_C = \{1, 2\}$ are the blocks of $\vect{d}_{i,aa}^p$ for $1 < i < p$, and
\item for all other $\sigma^p$-blocks $C$ of $A^p$, $\range f_p|_C = \{0\}$.
\end{compactitem}
This implies that there exists a map $\tau \colon \{1, \dotsc, n\} \to \{1, \dotsc, m\}$ such that $\tau(1) = 1$, $\tau(n) = m$ and for every $i$ ($1 \leq i \leq n$) and $\alpha, \beta \in \{a, b\}$, it holds that $\vect{h}(\vect{d}_{i, \alpha \beta}^n) \in \{\vect{d}_{\tau(i), \alpha a}^m, \vect{d}_{\tau(i), \alpha b}^m\}$.

It is easy to verify that for all $p \geq 3$, $(\vect{d}_{i,aa}, \vect{d}_{i,ba}, \vect{d}_{j,aa}) \in \rho^p$ if and only if $i = j$ or $i = j + 1$. Since for every $i$ ($1 \leq i \leq n - 1$), $(\vect{d}_{i+1,aa}^n, \vect{d}_{i+1,ba}^n, \vect{d}_{i,aa}^n) \in \rho^{n+1}$, it follows that
\[
(\vect{d}_{\tau(i+1),a \beta_1}^m, \vect{d}_{\tau(i+1),b \beta_2}^m, \vect{d}_{\tau(i),a \beta_3}^m) = \bigl( \vect{h}(\vect{d}_{i+1,aa}^n), \vect{h}(\vect{d}_{i+1,ba}^n), \vect{h}(\vect{d}_{i,aa}^n) \bigr) \in \rho^m
\]
for some $\beta_1, \beta_2, \beta_3 \in \{a, b\}$. By the previous observation, $\tau(i + 1) \in \{\tau(i), \tau(i) + 1\}$. Since $\tau(1) = 1$, this implies that $\tau(n) \leq n < m = \tau(n)$, and we have reached the desired contradiction.
\end{proof}

\begin{lemma}
\label{ClonesCase40}
Let $A = \nset{3} = \{a, b, c\}$. For the relation
\[
\rho = \begin{pmatrix}
a & b & a & c & a & b & a & a & b & b & a & b & c & a & a & c & c & a & c \\
b & a & c & a & a & a & b & a & b & a & b & b & a & c & a & c & a & c & c \\
c & c & b & b & a & a & a & b & a & b & b & b & a & a & c & a & c & c & c
\end{pmatrix}
\]
in line 40 of Table \ref{table-submaxcl3}, $\Pol \rho \notin \FF_\nset{3}$.
\end{lemma}

\begin{proof}
Let $\cl{C} = \Pol \rho$. For $n \geq 3$, $1 \leq i \leq n - 1$, denote by $\vect{a}_i^n$ the $n$-tuple satisfying
\[
\vect{a}_i^n(i) = b, \qquad
\vect{a}_i^n(i - 1) = \vect{a}_i^n(i + 1) = a, \qquad
\vect{a}_i^n(j) = c \quad (j \notin \{i - 1, i, i + 1\}),
\]
where addition is done modulo $n$, i.e., $n + 1 \equiv 1$, $0 \equiv n$.

For $n \geq 3$, define the operation $f_n \colon A^n \to A$ as follows:
\[
f_n(\vect{a}) =
\begin{cases}
1, & \text{if $\vect{a} = \vect{a}_i^n$ for some $i \in \{1, \dotsc, n\}$,} \\
2, & \text{if $\vect{a} = \bar{c}$,} \\
0, & \text{otherwise.}
\end{cases}
\]
We claim that if $n$ and $m$ are distinct odd positive integers, then $f_n \not\fequiv[\cl{C}] f_m$, and hence there are infinitely many $\fequiv[\cl{C}]$-classes. For, let $n < m$ and assume on the contrary that there exists a map $\vect{h} = (h_1, \dotsc, h_m) \in \cl{C}^m$ such that $f_n = f_m \circ \vect{h}$. Then $\vect{h}(\bar{c}) = \bar{c}$ and there exists a map $\tau \colon \{1, \dotsc, n\} \to \{1, \dotsc, m\}$ such that $\vect{h}(\vect{a}_i^n) = \vect{a}_{\tau(i)}^m$.

It is easy to verify that for any $p$, $(\vect{a}_i^p, \vect{a}_j^p, \bar{c}) \in \rho^p$ if and only if $i = j + 1$ or $i = j - 1$ (where addition is done modulo $p$). Since for every $i \in \{1, \dotsc, n\}$, $(\vect{a}_i^p, \vect{a}_{i+1}^p, \bar{c}) \in \rho^n$ (addition modulo $n$), we have that $(\vect{a}_{\tau(i)}^m, \vect{a}_{\tau(i+1)}^m, \bar{c}) = \bigl(\vect{h}(\vect{a}_i^n), \vect{h}(\vect{a}_{i+1}^n), \vect{h}(\bar{c}) \bigr) \in \rho^m$ (addition modulo $n$ and $m$, respectively). By the previous observation, $\tau(i + 1) \in \{\tau(i) - 1, \tau(i) + 1\}$ (addition modulo $n$ and $m$, respectively). It is then easy to verify that whenever $n$ and $m$ are odd integers and $n < m$, it is not possible to have such a map $\tau$ (for, $\tau$ cannot be surjective, and thus the preimages of each $j \in \{1, \dotsc, m\}$ have the same parity). We have reached the desired contradiction.
\end{proof}

\section*{Acknowledgements}

We would like to thank Miguel Couceiro for helpful discussions concerning the clones in lines 34 and 35 of Table \ref{table-submaxcl3}.

\end{document}